\documentclass{article}
\usepackage{amsmath}
\usepackage{amsfonts}
\usepackage{amssymb}
\usepackage{graphicx}

\newtheorem{theorem}{Theorem}[section]

\newtheorem{corollary}[theorem]{Corollary}

\newtheorem{lemma}[theorem]{Lemma}

\newtheorem{remark}[theorem]{Remark}

\numberwithin{equation}{section}
\newenvironment{proof}[1][Proof]{\noindent\textbf{#1.} }{\ \rule{0.5em}{0.5em}}
\begin{document}
\title{On a sharp volume estimate for gradient Ricci solitons with scalar curvature bounded below}

\author{Shijin Zhang \thanks{Shijin Zhang is currently a visiting PhD student at the Department of Mathematics UCSD. He is partially supported by China Scholarship Council.}\\Chern Institute of Mathematics, Nankai University,\\Tianjin, 300071, P.R.China\\Department of Mathematics,University of California
at San Diego, \\La Jolla, CA 92093,USA\\shijin\_zhang@yahoo.com}
\maketitle

\begin{abstract}
In this note, we obtain a sharp volume estimate for complete gradient Ricci
solitons with scalar curvature bounded below by a positive constant.
Using Chen-Yokota's argument we obtain a local lower bound estimate of
the scalar curvature for the Ricci flow on complete
manifolds. Consequently, one has a sharp estimate of
the scalar curvature for expanding Ricci solitons; we also
provide a direct (elliptic) proof of this sharp estimate.
Moreover, if the scalar curvature attains its minimum
value at some point, then the manifold is Einstein.
\end{abstract}

\section*{Introduction}
The Ricci flow $\frac{\partial}{\partial t}g(x,t)=-2Ric(x,t)$, was
introduced by Hamilton in \cite{H}. We say that a quadruple
$(M^{n},g,f,\varepsilon)$, where $(M^{n},g)$ is a Riemannian
manifold, $f$ is a smooth function on $M^{n}$ and
$\varepsilon\in\mathbb{R}$, is a gradient Ricci soliton if
\begin{equation}
R_{ij}+\nabla_{i}\nabla_{j}f+\frac{\varepsilon}{2}g_{ij}=0.
\end{equation}
We call $f$ the potential function. We say that $g$ is shrinking,
steady, or expanding if $\varepsilon < 0$, $\varepsilon=0$, or
$\varepsilon>0$, respectively.

The following volume growth estimate for complete shrinking gradient
Ricci solitons was proved by O. Munteanu \cite{M}, with an important special case was proved
by H.-D. Cao and D.-T. Zhou \cite{CZ}. Let
$(M^{n},g,f,-1)$ be a complete shrinking gradient Ricci soliton.
Given $o\in M^{n}$, there exists a constant $C<\infty$ such that
\begin{center}
$V(B(o,r))\leq C(r+1)^{n}$
\end{center}
for all $r\geq 0$, where $B(o,r)$ is the ball of radius $r$ at
center $o$ and $V(B(o,r))$ denotes the volume of $B(o,r)$ with
respect to the metric $g$. From the proof of Proposition 2.1 in
\cite{CN}, we can obtain the following property; see Lemma 1.1
below.
\begin{lemma}
Let $(M^{n},g,f,-1)$ be a complete gradient shrinking Ricci soliton
with $R\geq \delta >0$. Then for any $\eta>0$, there exists a constant
$C_{1}<\infty$ depending on $\eta$ and the soliton such that
\begin{center}
$V(B(o,r))\leq C_{1}(r+1)^{n-(2-\eta)\delta}$.
\end{center}
for all $r>0$.
\end{lemma}
We sharpen the above result as follows, which is our main
theorem.
\begin{theorem}\label{MainThm}
Let $(M^{n},g,f,-1)$ be a complete shrinking gradient Ricci soliton
with $R\geq \delta> 0$. Then given $o\in M^{n}$, there exists a constant
$C<\infty$ depending only on $\delta$, $o$ and the soliton such that
\begin{center}
$V(B(o,r))\leq C(r+1)^{n-2\delta}$
\end{center}
for all $r\geq 0$.
\end{theorem}

\begin{remark}
Above result is sharp. For example, a product $M^n=N^k\times \mathbb{R}^{n-k}$($k=2,3,\cdots,n$), where $N^k$ is an
Einstein manifold with constant scalar curvature $\frac{k}{2}$ and take $f=\frac{|x|^2}{4}$ on $\mathbb{R}^{n-k}$, then
the equality in Theorem \ref{MainThm} holds.
\end{remark}

The following property for gradient Ricci solitons is the second
part of Theorem 1.3 in Z.-H. Zhang \cite{Z}. In fact, part 1 is a consequence of Corollary 2.5 in B.-L. Chen \cite{C1}.
\begin{theorem}
Let $(M^{n},g,f,\varepsilon)$ be a noncompact complete gradient
Ricci soliton.
\begin{enumerate}
\item If the gradient soliton is shrinking or steady, then $R\geq 0$.
\item If the gradient soliton is expanding, then there exists a positive constant $C(n)$ such that $R\geq -C(n)\varepsilon$.
\end{enumerate}
\end{theorem}

The following property is an improvement of part 2 of Theorem 0.4, which is the sharp estimate for
noncompact expanding gradient Ricci solitons. The compact case follows
from a direct application of the maximum principle; we know the
manifold is Einstein (see Proposition 9.43 in \cite{CLN}).
\begin{theorem}
\label{M1} Let $(M^{n},g,f,1)$ be a complete expanding gradient
Ricci soliton. Then $R\geq -\frac{n}{2}$. Furthermore, if there
exists a point $x_{0}\in M^{n}$ such that $R(x_{0})=-\frac{n}{2}$,
then $(M^{n},g)$ is Einstein, i.e., $R_{ij}=-\frac{1}{2}g_{ij}$.
\end{theorem}
The first part of Theorem \ref{M1} is a consequence of Corollary 2.3 (i) in B.-L. Chen \cite{C1}(see Corollary 2.2 below); we also provide a direct (elliptic) proof of this sharp estimate.

\section{Volume growth of complete noncompact gradient Ricci solitons}
We consider the complete shrinking gradient Ricci solitons in this
section, i.e. $\varepsilon=-1$. Normalizing $f$, from (0.1) we have
\begin{equation}
R+|\nabla f|^{2}-f\equiv 0.
\end{equation}
Define
\begin{center}
$\textbf{V}(c)\doteqdot \int_{\{f\leq c\}}d\mu=Vol\{f\leq c\}$,
\end{center}
\begin{center}
$\textbf{R}(c)\doteqdot \int_{\{f\leq c\}}Rd\mu$.
\end{center}
By the co-area formula
\begin{center}
$\textbf{V}'(c)=\int_{\{f=c\}}\frac{1}{|\nabla f|}d\sigma,$
\end{center}
\begin{center}
$\textbf{R}'(c)=\int_{\{f=c\}}\frac{R}{|\nabla f|}d\sigma.$
\end{center}
Since $R\geq 0$ (see Corollary 2.3 below), we have
\begin{center}
$\textbf{R}(c)\geq 0$ and $\textbf{R}'(c)\geq 0$.
\end{center}
Integrating $R+\Delta f=\frac{n}{2}$ over $\{f\leq c\}$ yields
\begin{equation}
\begin{aligned}
\frac{n}{2}\textbf{V}(c)-\textbf{R}(c)=&\int_{\{f\leq c\}}\Delta fd\mu\\
=&\int_{\{f=c\}}\frac{\partial f}{\partial \nu}d\sigma\\
=&\int_{\{f=c\}}|\nabla f|d\sigma.
\end{aligned}
\end{equation}
In particular,
\begin{equation}
\frac{n}{2}\textbf{V}(c)\geq \textbf{R}(c).
\end{equation}
Since by (1.1),
\begin{center}
$|\nabla f|=\frac{|\nabla f|^{2}}{|\nabla f|}=\frac{f-R}{|\nabla
f|}.$
\end{center}
So we have
\begin{equation}
\frac{n}{2}\textbf{V}(c)-\textbf{R}(c)=c\textbf{V}'(c)-\textbf{R}'(c).
\end{equation}
That is
\begin{equation}
\frac{n}{2}\textbf{V}(c)-c\textbf{V}'(c)=\textbf{R}(c)-\textbf{R}'(c).
\end{equation}
\begin{lemma}
Let $(M^{n},g,f,-1)$ be a complete gradient shrinking Ricci soliton
with $R\geq \delta >0$, then for any $\eta >0$, there exist $c_{0},
C_{1}$ depending on $\eta$ and $\delta$, when $c\geq c_{0}$, we have
\begin{center}
$\textbf{V}(c)\leq C_{1}c^{\frac{n-(2-\eta)\delta}{2}}$.
\end{center}
\end{lemma}
\begin{proof}
For the sake of completeness, we provide the detailed proof. The
proof is similar to the proof of Theorem 1 in \cite{M}(or the proof
of Proposition 2.1 in \cite{CN}). If $\eta \geq 2$, this has done by
Theorem 1 in \cite{M}.

So we only consider $\eta <2$. Now using the positive lower bound
for $R$ and $\eta <2$, we have
\begin{equation}
\begin{aligned}
\frac{n-(2-\eta)\delta}{2}\textbf{V}(c)-\frac{\eta}{2}\textbf{R}(c)
\geq&
\frac{n}{2}\textbf{V}(c)-\textbf{R}(c)\\=&c\textbf{V}'(c)-\textbf{R}'(c).
\end{aligned}
\end{equation}
This implies
\begin{eqnarray*}
\frac{d}{dc}(c^{-\frac{n-(2-\eta)\delta}{2}}\textbf{V}(c))=&c^{-\frac{n+2-(2-\eta)\delta}{2}}(c\textbf{V}'(c)-\frac{n-(2-\eta)\delta}{2}\textbf{V}(c))\\
\leq&
c^{-\frac{n+2-(2-\eta)\delta}{2}}(\textbf{R}'(c)-\frac{\eta}{2}\textbf{R}(c)).
\end{eqnarray*}
Integrating this by parts on $[c_{0},\overline{c}]$ yields
\begin{eqnarray*}
\overline{c}^{-\frac{n-(2-\eta)\delta}{2}}\textbf{V}(\overline{c})-c_{0}^{-\frac{n-(2-\eta)\delta}{2}}\textbf{V}(c_{0})\leq&
\overline{c}^{-\frac{n+2-(2-\eta)\delta}{2}}\textbf{R}(\overline{c})-c_{0}^{-\frac{n+2-(2-\eta)\delta}{2}}\textbf{R}(c_{0})\\
&+
\int_{c_{0}}^{\overline{c}}(\frac{n+2-(2-\eta)\delta}{2}-\frac{\eta
c}{2})c^{-\frac{n+4-(2-\eta)\delta}{2}}\textbf{R}(c)dc.
\end{eqnarray*}
Since $\textbf{R}(c)\geq 0$, for $c_{0}\geq
\frac{n+2-(2-\eta)\delta}{\eta}$ we have
\begin{eqnarray*}
\overline{c}^{-\frac{n-(2-\eta)\delta}{2}}\textbf{V}(\overline{c})-c_{0}^{-\frac{n-(2-\eta)\delta}{2}}\textbf{V}(c_{0})\leq&
\overline{c}^{-\frac{n+2-(2-\eta)\delta}{2}}\textbf{R}(\overline{c})\\
\leq&
\frac{n}{2}\overline{c}^{-\frac{n+2-(2-\eta)\delta}{2}}\textbf{V}(\overline{c})
\end{eqnarray*}
the last inequality has used (1.3). Thus if $\overline{c}\geq
\max\{n,c_{0}\}, c_{0}\geq \frac{n+2-(2-\eta)\delta}{\eta}$, then
\begin{center}
$\textbf{V}(\overline{c})\leq
2c_{0}^{-\frac{n-(2-\eta)\delta}{2}}\textbf{V}(c_{0})\overline{c}^{\frac{n-(2-\eta)\delta}{2}}.$
\end{center}
So Lemma holds.
\end{proof}
\begin{theorem}
Let $(M^{n},g,f,-1)$ be a complete gradient shrinking Ricci soliton
with $R\geq \delta >0$, then there exists a positive constant $C$
depending only on $\delta$, $o$ and the soliton such that
\begin{center}
$V(B(o,r))\leq C(1+r)^{n-2\delta}$.
\end{center}
\end{theorem}
\begin{proof}
Since $R\geq \delta, \textbf{R}(c)\geq 0$, we have
\begin{eqnarray*}
\frac{n-2\delta}{2}\textbf{V}(c)\geq&
\frac{n}{2}\textbf{V}(c)-\textbf{R}(c)\\
=&c\textbf{V}'(c)-\textbf{R}'(c).
\end{eqnarray*}
Since $\textbf{R}'(c)\geq 0$, we have
\begin{eqnarray*}
\frac{d}{dc}(c^{-\frac{n-2\delta}{2}}\textbf{V}(c))=&c^{-\frac{n+2-2\delta}{2}}(c\textbf{V}'(c)-\frac{n-2\delta}{2}\textbf{V}(c))\\
\leq& c^{-\frac{n+2-2\delta}{2}}\textbf{R}'(c).
\end{eqnarray*}
Integrating this by parts on $[c_{0},\overline{c}]$ yields
\begin{equation}
\begin{aligned}
\overline{c}^{-\frac{n-2\delta}{2}}\textbf{V}(\overline{c})-c_{0}^{-\frac{n-2\delta}{2}}\textbf{V}(c_{0})\leq&
\int_{c_{0}}^{\overline{c}}c^{-\frac{n+2-2\delta}{2}}\textbf{R}'(c)dc\\
=&\overline{c}^{-\frac{n+2-2\delta}{2}}\textbf{R}(\overline{c})-c_{0}^{-\frac{n+2-2\delta}{2}}\textbf{R}(c_{0})
\\+&\frac{n+2-2\delta}{2}\int_{c_{0}}^{\overline{c}}c^{-\frac{n+4-2\delta}{2}}\textbf{R}(c)dc.
\end{aligned}
\end{equation}
By (1.3), we have
\begin{center}
$\int_{c_{0}}^{\overline{c}}c^{-\frac{n+4-2\delta}{2}}\textbf{R}(c)dc\leq
\frac{n}{2}\int_{c_{0}}^{\overline{c}}c^{-\frac{n+4-2\delta}{2}}\textbf{V}(c)dc$.
\end{center}
Let $\eta =\frac{1}{\delta}$ in Lemma 1.1, so when $c$ is large enough,
we have
\begin{center}
$\textbf{V}(c)\leq C_{1}c^{\frac{n+1-2\delta}{2}}$.
\end{center}
So
\begin{eqnarray*}
\int_{c_{0}}^{\overline{c}}c^{-\frac{n+4-2\delta}{2}}\textbf{R}(c)dc\leq&
\frac{nC_{1}}{2}\int_{c_{0}}^{\overline{c}}c^{-\frac{n+4-2\delta}{2}}c^{\frac{n+1-2\delta}{2}}dc\\
=&\frac{nC_{1}}{2}\int_{c_{0}}^{\overline{c}}c^{-\frac{3}{2}}dc\\
=& nC_{1}(c_{0}^{-\frac{1}{2}}-\overline{c}^{-\frac{1}{2}})\\
\leq& nC_{1}c_{0}^{-\frac{1}{2}}.
\end{eqnarray*}
Since (1.3) and $\textbf{V}(c)\geq 0$, $\delta \leq \frac{n}{2}$. So
\begin{equation}
\overline{c}^{-\frac{n-2\delta}{2}}\textbf{V}(\overline{c})-c_{0}^{-\frac{n-2\delta}{2}}\textbf{V}(c_{0})\leq
\frac{n}{2}\overline{c}^{-\frac{n+2-2\delta}{2}}\textbf{V}(\overline{c})+\frac{n+2-2\delta}{2}nC_{1}c_{0}^{-\frac{1}{2}}.
\end{equation}
Then same argument in the proof of Lemma 1.1, when $\overline{c}$ is
large enough, there exists a constant $C_{2}$ depending only on
$\delta$ such that
\begin{equation}
\textbf{V}(\overline{c})\leq
C_{2}\overline{c}^{\frac{n-2\delta}{2}}.
\end{equation}
By Theorem 1.1 of \cite{CZ} (or see \cite{C}), there exists a constant $C$ depending
only on $g$ and $o$ such that
\begin{equation}
\frac{1}{4}(r(x)-C)^{2} \leq f(x) \leq \frac{1}{4}(r(x)+C)^{2}
\end{equation}
where $r(x)$ denotes the distance from $x$ to $o$. Hence we obtain
the result.
\end{proof}

\section{Lower bound of scalar curvature for Ricci flow}
In this section, we observe that by a modification of B.-L.
Chen's theorem (Corollary 2.3(i) in \cite{C1}), we obtain a local lower bound estimate of the scalar curvature for the Ricci flow on complete
manifolds, we follow Yokota's argument
in Proposition A.3 in \cite{Y}.

\begin{theorem}
For any $0<\varepsilon<\frac{2}{n}$. Suppose $(M^{n},g(t)),\ \ t\in [\alpha,\beta]$ is a complete
solution to Ricci flow, $p\in M$, then there exist constants $C(p)$ depending on $p$ and the metrics $g(t) (t\in[\alpha,\beta])$ and $C$ such that when $c\geq C(p)$, we have
\begin{equation}
R(x,t)\geq -B\frac{e^{2AB(t-\alpha)}+1}{e^{2AB(t-\alpha)}-1}
\end{equation}
whenever $x\in B_{g(t)}(p,c), t\in (\alpha,\beta]$, where $A(\varepsilon)=\frac{2}{n}-\varepsilon,B(\varepsilon)=\frac{3C}{2\sqrt{A\varepsilon}c^{2}}$.
\end{theorem}
\begin{proof}
First we use the cutoff function in the proof of Proposition A.3
in \cite{Y}. Let $\eta:\mathbb{R}\rightarrow [0,1]$ be a
nonincreasing $\mathcal{C}^{2}$ function such that $|\eta'|$ and
$|\eta''|$ are bounded and $\eta(u)=1$ for any $u\in(-\infty,1]$,
$\eta(u)=0$ for any $u\in[2,\infty)$ and $\eta(u)=(2-u)^{4}$ for any
$u \in[\frac{3}{2},2]$. Then there exists a positive constant $C$
such that
\begin{eqnarray}
\frac{(\eta'(u))^{2}}{\eta(u)}\leq C\eta(u)^{\frac{1}{2}},\\
|\eta''(u)|\leq C\eta(u)^{\frac{1}{2}}.
\end{eqnarray}

Clearly we can choose a number $r_{0}\in (0,1)$, such that
\begin{center}
$Rc(g(t))\leq (n-1)r_{0}^{-2}$
\end{center}
in $B_{g(t)}(p,r_{0})$ for $t\in [\alpha,\beta]$. Let $C(p)=
r_{0}+\frac{5}{3}(n-1)r_{0}^{-1}(\beta-\alpha)$ and given any $c\geq C(p)$.

Given any time $t_{0}\in (\alpha,\beta]$ and suppose
that \begin{equation} R(p,t_{0})<0.
\end{equation}
Define
$Q:M\times[\alpha,t_{0}]\rightarrow \mathbb{R}$ by
\begin{equation}
Q(x,t)=\eta(\frac{\tilde{r}(x,t)}{c})R(x,t),
\end{equation}
where \begin{equation} \tilde{r}(x,t)\doteqdot
d_{g(t)}(x,p)+\frac{5}{3}(n-1)r_{0}^{-1}(t_{0}-t). \end{equation}
Then $Q(x,t)$ is a compactly support function.

By Lemma 8.3 (a) in \cite{P}, we have
\begin{equation}
(\frac{\partial}{\partial t}-\Delta)\tilde{r}(x,t)\geq 0
\end{equation}
whenever $d_{g(t)}(x,p)> r_{0},\ \ t\in [\alpha,t_{0}]$, in the
barrier sense. We have
\begin{eqnarray*}
(\frac{\partial}{\partial t}-\Delta)Q=&\eta(\frac{\partial}{\partial
t}-\Delta)R+\frac{\eta'}{c}R(\frac{\partial}{\partial
t}-\Delta)\tilde{r}\\
&-\frac{2}{c}\eta'<\nabla \tilde{r},\nabla R>-\frac{\eta''}{c^{2}}R
\end{eqnarray*}
where $\eta$ denotes $\eta(\frac{\tilde{r}}{c})$. In the case of
$d_{g(t)}(x,p)\leq r_{0}$, then $\tilde {r}(x,t)\leq c$, so at point
$(x,t)$, $\eta =1,\ \ \eta'=\eta''=0$, so
\begin{equation}
(\frac{\partial}{\partial t}-\Delta)Q=(\frac{\partial}{\partial
t}-\Delta)R.
\end{equation}
In the case of $d_{g(t)}(x,p)>r_{0}$,
then we applying (2.7) and $\eta'\leq 0$, we have at a point where
$R\leq 0$
\begin{equation}
(\frac{\partial}{\partial t}-\Delta)Q\geq
\eta\frac{2}{n}R^{2}-\frac{2\eta'}{c\eta}<\nabla \tilde{r},\nabla
Q>+\frac{1}{c^{2}}(2\frac{(\eta')^{2}}{\eta}-\eta'')R
\end{equation}
whenever $\eta \neq0$. Hence by both of cases, we have at point
$(x,t)$ where $R(x,t)\leq 0$, (2.9) holds whenever $\eta\neq0$.
Applying (2.2) and (2.3) to (2.9), we have at any point where $Q<0$
\begin{equation}
(\frac{\partial}{\partial t}-\Delta)Q\geq
\eta\frac{2}{n}R^{2}-\frac{2\eta'}{c\eta}<\nabla \tilde{r},\nabla
Q>+\frac{3C}{c^{2}}\eta^{\frac{1}{2}}R.
\end{equation}

Let $Q_{m}(t)=\min_{x\in M}Q(x,t)$. Then
\begin{center}
$Q_{m}(t_{0})\leq R(p,t_{0})<0$.
\end{center}
By (2.10) we have for any $t\in [\alpha,t_{0}]$ where $Q_{m}(t)<0$
and for any $x_{t}$ such that $Q(x_{t},t)=Q_{m}(t)$, then for any
$\varepsilon\in (0,\frac{2}{n})$
\begin{eqnarray*}
\frac{d_{-}}{dt}Q_{m}(t)\geq& \frac{2}{n}\eta
R(x_{t},t)^{2}+\frac{3C}{c^{2}}\eta^{\frac{1}{2}}R(x_{t},t)\\
\geq&(\frac{2}{n}-\varepsilon)\eta
R(x_{t},t)^{2}-\frac{9C^{2}}{4\varepsilon C^{4}}\\
\geq&(\frac{2}{n}-\varepsilon)Q_{m}^{2}-\frac{9C^{2}}{4\varepsilon
c^{4}}
\end{eqnarray*}
using $ab\geq -\varepsilon a^{2}-\frac{1}{4\varepsilon}b^{2}$ and
$0<\eta\leq 1$.

Recall that the solution of ODE
\begin{eqnarray*}
\begin{cases}
\frac{dq}{dt}&=A(q^{2}-B^{2}),\\
q(t_{0})&=q_{0}
\end{cases}
\end{eqnarray*}
on $[\alpha,t_{0}]$,
then
\begin{eqnarray*}
q(t)=\begin{cases}
-B\frac{De^{-2AB(t_{0}-t)}+1}{De^{-2AB(t_{0}-t)}-1} & \text{if $B\neq -q_{0}$},\\
q_{0} & \text{if $B=-q_{0}$}
\end{cases}
\end{eqnarray*}
where $D=\frac{q_{0}-B}{q_{0}+B}$ provided $B\neq -q_{0}$.

Taking $A=\frac{2}{n}-\varepsilon,\ \ B=\frac{3C}{2\sqrt{A\varepsilon}c^{2}}$, and $q_{0}=Q_{m}(t_{0})<0$, then we have
\begin{center}
$Q_{m}(t)\leq q(t).$
\end{center}
$q(t)>-\infty$ for $t\in [\alpha,t_{0}]$, since $Q_{m}(t)>-\infty$ for $t\in [\alpha,t_{0}]$.

\textbf{Case (1).} If $q_{0}\geq -B$, then we have
\begin{equation}
R(x,t_{0})\geq Q_{m}(t_{0})\geq -B
\end{equation}
whenever $x\in B_{g(t_{0})}(p,c)$, since $\eta=1$.

\textbf{Case (2).} If $q_{0}<-B$, then $D>1$ and since
$q(t)>-\infty$ for $t\in [\alpha,t_{0}]$, we have
\begin{center}
$De^{-2AB(t_{0}-t)}-1>0$
\end{center}
for all $t\in [\alpha,t_{0}]$, so
\begin{center}
$De^{-2AB(t_{0}-\alpha)}-1>0$ i.e., $q_{0}>
-B\frac{e^{2AB(t_{0}-\alpha)}+1}{e^{2AB(t_{0}-\alpha)}-1}$,
\end{center}
 so that
\begin{equation}
R(x,t_{0})\geq q_{0}>
-B\frac{e^{2AB(t_{0}-\alpha)}+1}{e^{2AB(t_{0}-\alpha)}-1}
\end{equation}
whenever $x\in B_{g(t_{0})}(p,c)$.

Since $B>0$, (2.12) is a better estimate, we conclude that (2.12)
holds in either case.

Since $c$ independent of $t$, we complete the proof of this theorem.
\end{proof}

Hence we also obtain a consequence of Corollary 2.3(i) in B.-L. Chen \cite{C1}.
\begin{corollary}
Suppose $(M^{n},g(t)),\ \ t\in[\alpha,\beta]$, is a complete solution to Ricci flow, then
\begin{equation}
R\geq -\frac{n}{2(t-\alpha)}
\end{equation}
on $M\times (\alpha,\beta]$.
\end{corollary}
\begin{proof}
For any $t_{0}\in (\alpha,\beta]$. Now fix $\varepsilon\in (0,\frac{2}{n})$ and let $c\rightarrow
\infty$. Then $B\rightarrow 0$. Since
\begin{equation*}
\lim_{B\rightarrow 0}B\frac{e^{2AB(t_{0}-\alpha)}+1}{e^{2AB(t_{0}-\alpha)}-1}=\frac{1}{A(t_{0}-\alpha)}
\end{equation*}
and (2.12) independent of $c$, we obtain
\begin{center}
$R(x,t_{0})\geq -\frac{1}{(\frac{2}{n}-\varepsilon)(t_{0}-\alpha)}.$
\end{center}
on $M\times \{t_{0}\}$. Finally, taking $\varepsilon\rightarrow0$, we obtain
\begin{equation*}
R(x,t_{0})\geq -\frac{n}{2(t_{0}-\alpha)}.
\end{equation*}
on $M\times \{t_{0}\}$.
Since above argument holds for any $t\in (\alpha,\beta]$, we obtain the corollary.
\end{proof}

The following property is Corollary 2.5 in \cite{C1}( or Proposition A.3 in \cite{Y}).
\begin{corollary}
If $(M^{n},g(t)),\ \ t\in (-\infty,0],$ is a complete ancient
solution to the Ricci flow, then
\begin{center}
$R\geq 0$
\end{center}
on $M\times (-\infty,0]$.
\end{corollary}

\begin{corollary}
Suppose $(M^{n},g,f,\varepsilon)$ be a noncompact complete gradient
Ricci soliton. Then
\begin{enumerate}
\item If the gradient soliton is shrinking or steady, then $R\geq 0$.
\item If the gradient soliton is expanding, then $R\geq -\frac{n\varepsilon}{2}$. Moreover, if the
scalar curvature attain the minimum value $-\frac{n\varepsilon}{2}$
at some point, then $(M^{n},g(t))$ is Einstein.
\end{enumerate}
\end{corollary}
\begin{proof}
Part (1) is a consequence of Corollary 2.3.

As shown in Theorem 4.1 of \cite{CLN}, associated to the metric and
the the potential function $f$, there exists a family of metrics
$g(t)$, a solution to Ricci flow $\frac{\partial}{\partial
t}g(t)=-2Ric(g(t))$, with the property that $g(0)=g$, and a family
of diffeomorphisms $\phi(t)$, which is generated by the vector field
$X=\frac{1}{\tau}\nabla f$, such that $\phi(0)=id$, and
$g(t)=\tau(t)\phi^{*}(t)g$ with $\tau(t)=1+\varepsilon t>0$, as well
as $f(x,t)=\phi^{*}(t)f(x)$.

For expanding gradient Ricci soliton, i.e. $\varepsilon>0$. We know
$t\in (-\frac{1}{\varepsilon},\infty)$, so by Corollary 2.2 we obtain
$R(x,t)\geq -\frac{n}{2(t+\frac{1}{\varepsilon})}$, i.e., $R(x,t)\geq
-\frac{n\varepsilon}{2\tau}$.

Let
$\tilde{R}=R+\frac{n\varepsilon}{2\tau}$, so $\tilde{R}\geq 0$ and
\begin{equation}
\begin{aligned}
\frac{\partial}{\partial t}\tilde{R}=&\Delta
R+2|Rc|^2-\frac{n\varepsilon^{2}}{2\tau^{2}}\\
=&\Delta\tilde{R}+2|Rc-\frac{R}{n}g|^2+\frac{2R^2}{n}-\frac{n\varepsilon^{2}}{2\tau^{2}}\\
=&\Delta
\tilde{R}+2|Rc-\frac{R}{n}g|^2+\frac{2}{n}\tilde{R}(\tilde{R}-\frac{n\varepsilon}{\tau})\\
\geq&\Delta\tilde{R}-\frac{2\varepsilon}{\tau}\tilde{R}.
\end{aligned}
\end{equation}
So
\begin{equation}
\frac{\partial}{\partial t}(\tau^2\tilde{R})\geq \Delta
(\tau^2\tilde{R}).
\end{equation}
By strong maximum principle (Theorem 6.54 in \cite{CLN}), we know
that if there exists a point $x_{0}$ such that
$\tilde{R}(x_{0},t_{0})=0$ for some $t_{0}>-\frac{1}{\varepsilon}$,
then $\tilde{R}(x,t_{0})\equiv 0$ for all $x\in M$. So
\begin{center}
$R(x,t_{0})\equiv -\frac{n\varepsilon}{2\tau(t_{0})}$.
\end{center}
So when $t_{0}=0$, i.e. $\tau(0)=1$, then
$R(x,0)\equiv-\frac{n\varepsilon}{2}$. From Lemma 3.1 (1) below (or see section 4.1 in \cite{CLN}), we have
\begin{equation}
\Delta R+2|Rc|^{2}-<\nabla f,\nabla R>+\varepsilon R=0.
\end{equation}
So by (2.16) we know
\begin{center}
$|Rc+\frac{\varepsilon}{2}g|^{2}=0$. \end{center} So
\begin{center}
$R_{ij}=-\frac{\varepsilon}{2}g_{ij}.$
\end{center}
Hence $(M,g,f,\varepsilon)$ is Einstein.
\end{proof}

In next section, we will provide a direct (elliptic) proof of the first part of
Theorem \ref{M1}.

\section{Direct proof for expanding Ricci solitons}
In this section, we provide a direct (elliptic) proof of lower bound for
scalar curvature for complete expanding gradient Ricci solitons.  We use a cutoff function argument to equation (0.1).
\begin{lemma}
Let $(M^{n},g,f,\varepsilon)$ be a complete gradient Ricci soliton.
Fix $o\in M^{n}$, and define $r(x)\doteqdot d(x,o)$, then the
following
hold
\begin{enumerate}
\item $\Delta R+2|Rc|^{2}-<\nabla f,\nabla R>+\varepsilon R=0.$
\item Suppose $Ric\leq (n-1)K$ on $B(o,r_{0})$, for some positive
numbers $r_{0}$ and $K$. Then for any point $x$, outside
$B(o,r_{0})$
\begin{center}
$(\Delta r-<\nabla f,\nabla r>)(x)\leq -<\nabla f,\nabla
r>(o)+\frac{\varepsilon}{2}r(x)+(n-1)\{\frac{2}{3}Kr_{0}+r_{0}^{-1}\}$.
\end{center}
\end{enumerate}
\end{lemma}
Part (1) is well known. Part (2) follows from an idea of Perelman;
see Lemma 8.3 in \cite{P} and its antecedent in \S17 on 'Bounds on
changing distances', in \cite{H1}. For the detailed proof of part 2, also see \cite{C}.

Now we prove the first part of Theorem \ref{M1}.

\begin{proof}
For expanding gradient Ricci solitons, let $\varepsilon=1$, so that
(1) in Lemma 1.1 is
\begin{equation}
\Delta R+2|Rc|^{2}-<\nabla f,\nabla R>+R=0.
\end{equation}
If $M^{n}$ were closed, by Proposition 9.43 in \cite{CLN}, we know
$(M^{n},g,f,1)$ is Einstein. So we only consider the noncompact
case. Fix $o\in M^{n}$ and fix a large number $b$. Let
\begin{center}
$\eta:[0,\infty)\rightarrow [0,1]$
\end{center}
be a $\mathcal{C}^{\infty}$ nonincreasing cutoff function with
$\eta(u)=1$ for $u\in [0,1]$ and $\eta(u)=0$ for $u\in
[1+b,\infty)$. Define $\Phi : M\rightarrow \mathbb{R}$ by
\begin{center}
$\Phi(x)=\eta(\frac{r(x)}{c})R(x)$
\end{center}
for $c\in (0,\infty)$. Later we shall take $c\rightarrow \infty$.

We have
\begin{center}
$\Delta\Phi=\eta\Delta R+\frac{2\eta'}{c}<\nabla r,\nabla
R>+(\frac{\eta'}{c}\Delta r+\frac{\eta''}{c^{2}})R.$
\end{center}
We have dropped $'\circ\frac{r}{c}'$ in our notation. By (3.1), we
have
\begin{equation}
\begin{aligned}
\Delta \Phi =&\eta(-2|Rc|^{2}+<\nabla f, \nabla
R>-R)+\frac{2}{c}\eta'<\nabla r, \nabla
R>+(\frac{\eta'}{c}\Delta r+\frac{\eta''}{c^{2}})R\\
=&<\nabla f,\nabla \Phi>+\frac{2}{c}\frac{\eta'}{\eta}<\nabla r,
\nabla
\Phi>-(\frac{2}{c^{2}}\frac{(\eta')^{2}}{\eta}+\frac{\eta'}{c}<\nabla
f,\nabla r>)R\\
&+\eta(-2|Rc|^2-R)+(\frac{\eta'}{c}\Delta r+\frac{\eta''}{c^2})R
\end{aligned}
\end{equation}
at all points where $\eta\neq 0$.

Suppose $x_{0}\in M$ is such that
\begin{equation}
\Phi(x_{0})=\min_{M}\Phi<0.
\end{equation}
Since $R(x_{0})<0$, at $x_{0}$ we have
\begin{equation}
0\geq \eta(-\frac{2}{n}R-1)+\frac{\eta'}{c}(\Delta r-<\nabla
f,\nabla r>)+\frac{1}{c^2}(\eta''-2\frac{(\eta')^2}{\eta}).
\end{equation}
We consider two cases, depending on the location of $x_{0}$.

\textbf{Case (i).} Suppose $r(x_{0})<c$, so that
$\eta(\frac{r(x_{0})}{c})=1$ in a neighborhood of $x_{0}$. Then
(3.4) and (3.3) imply
\begin{eqnarray*}
0\geq& -\frac{2}{n}R(x_{0})-1\\
=&-\frac{2}{n}\Phi(x_{0})-1\\
\geq&-\frac{2}{n}\eta(\frac{r(x)}{c})R(x)-1
\end{eqnarray*}
for all $x\in M$. This implies the desired estimate
\begin{equation}
R(x)\geq -\frac{n}{2}
\end{equation}
for all $x\in B(o,c)$ since $\eta(\frac{r}{c})=1$ in $B(o,c)$.

\textbf{Case (ii).} Now suppose $r(x_{0})\geq c$ and again consider
(3.4). Note that we may choose $\eta$ so that
\begin{equation}
\eta''-2\frac{(\eta')^2}{\eta}\geq -C_{2}
\end{equation}
for some universal constant $C_{2}<\infty$. Since $\eta'\leq 0$,
applying Lemma 3.1 (2) and (3.6) to (3.4) yields for all $x\in M$
\begin{equation}
\begin{aligned}
\frac{2}{n}\Phi(x)\geq& \frac{2}{n}\Phi(x_{0})\\
\geq&\frac{\eta'(\frac{r(x_{0})}{c})}{c}(\frac{n-1}{r_{0}}-<\nabla
f,\nabla r>(o)+\frac{1}{2}r(x_{0})+\frac{2}{3}r_{0}
\max_{\overline{B(o,r_{0})}}Rc)
\\&-\eta(\frac{r(x_{0})}{c})-\frac{C_{2}}{c^2}
\end{aligned}
\end{equation}
where $C_{2}$ independent of $c$. Taking $r_{0}=1$ and $c\geq 2$, we
have for all $x\in B(o,c)$
\begin{equation}
\begin{aligned}
\frac{2}{n}R(x)\geq& \frac{\eta'(\frac{r(x_{0})}{c})}{c}(n-1+|\nabla
f|(o)+\frac{1}{2}r(x_{0})+\frac{2}{3}\max_{\overline{B(o,1)}}Rc)\\
&-\eta(\frac{r(x_{0})}{c})-\frac{C_{2}}{c^2}
\end{aligned}
\end{equation}
Since $-C_{2}\leq \eta'\leq 0$ imply that for all $x\in B(o,c)$
\begin{equation}
\begin{aligned}
\frac{2}{n}R(x)\geq& -\frac{C_{2}}{c}(n-1+|\nabla
f|(o)+\frac{2}{3}\max_{\overline{B(o,1)}}Rc+\frac{1}{c})\\
&+\frac{1}{2}\eta'(\frac{r(x_{0})}{c})\frac{r(x_{0})}{c}-\eta(\frac{r(x_{0})}{c}).
\end{aligned}
\end{equation}
When take $c\rightarrow \infty$, then the first term of right hand
side of (3.9) tends to $0$. So we only consider to estimate the term
$\frac{1}{2}\eta'(\frac{r(x_{0})}{c})\frac{r(x_{0})}{c}-\eta(\frac{r(x_{0})}{c})$.
Since $x_{0}\in B(o,(1+b)c)-B(o,c)$, we have $1\leq
\frac{r(x_{0})}{c}<1+b$. Define $h_{\eta}(u)$ by
\begin{center}
$h_{\eta}(u)=\frac{1}{2}\eta'(u)u-\eta(u)$.
\end{center}
So we only estimate $h_{\eta}(u)$ for $u\in [1,1+b]$.

If we replace $\eta$ with nonnegative piecewise linear function
$\theta(u)$ such that
\begin{eqnarray*}
\theta(u)=\begin{cases}
1 & \text{if $u\in [0,1]$},\\
\frac{1+b-u}{b} & \text{if $u\in [1,1+b]$},\\
0 & \text{if $u\in [1+b,\infty)$}
\end{cases}
\end{eqnarray*}
then $h_{\theta}(u)=-\frac{2b+2-u}{2b}$ for $u\in [1,1+b]$. So
$h_{\theta}(u)\geq -1$ for $u\in [2,b]$ and $h_{\theta}(u)\geq
-1-\frac{1}{b}$ for $u \in [1,2]$. For any small positive number
$\delta$, we can obtain a $\mathcal{C}^{\infty}$ cutoff function
$\beta$ after smooth the linear function $\theta$ such that
$\beta(u)=\theta(u)$ for $u\in [0,1]\cup[2,b]\cup[1+b,\infty)$ and
$-\frac{1+\delta}{b}\leq \beta'(u)\leq 0$ for $u\in
[1,2]\cup[b,1+b]$. So when $b$ is large and $\delta\leq \frac{b-1}{b+1}$, we have $h_{\beta}(u)\geq -1-\frac{1+\delta}{b}$ for $
u\in [1,1+b]$. Let $\eta$ equal
$\beta$, take $c\rightarrow \infty, \delta\rightarrow 0,
b\rightarrow \infty$, by (3.9) we obtain
\begin{center}
$\frac{2}{n}R(x)\geq -1$
\end{center}
for all $x\in M$. So
\begin{equation}
R(x)\geq -\frac{n}{2}
\end{equation}
for all $x\in M$.
\end{proof}
\section*{Acknowledgements}
The author would like to thank Professor Fuquan Fang and Professor
Lei Ni for their encouragement and constant help. He would also like
to thank Professor Ben Chow for encouragement and many helpful discussions and suggestions.
\bibliographystyle{amsplain}

\end{document}